\documentclass{amsart}
\usepackage{amsfonts}
\usepackage{amsmath,amssymb}
\usepackage{amsthm}
\usepackage{amscd}
\usepackage{graphics}
\usepackage{graphicx}
\theoremstyle{definition}{
\newtheorem{Def}{{\rm Definition}}
\newtheorem{Ex}{{\rm Example}}
\newtheorem{Rem}{{\rm Remark}}

}
\theoremstyle{plain}
{

\newtheorem{Prop}{Proposition}
\newtheorem{Thm}{Theorem}
\newtheorem{MainThm}{Main Theorem}

}

\begin{document}
\title[Some cohomologically similar manifolds  and special generic maps]{New families of manifolds with similar cohomology rings admitting special generic maps}
\author{Naoki Kitazawa}
\address{Institute of Mathematics for Industry, Kyushu University, 744 Motooka, Nishi-ku Fukuoka 819-0395, Japan\\
 TEL (Office): +81-92-802-4402 \\
 FAX (Office): +81-92-802-4405 \\
}
\email{n-kitazawa@imi.kyushu-u.ac.jp}
\urladdr{https://naokikitazawa.github.io/NaokiKitazawa.html}

\begin{abstract}
	As Reeb's theorem shows, Morse functions with exactly two singular points on closed manifolds are very simple and important. They characterize spheres whose dimensions are not $4$ topologically and the $4$-dimensional unit sphere.
	
{\it Special generic} maps are generalized versions of these maps. Canonical projections of unit spheres are special generic. Studies of Saeki and Sakuma since the 1990s, followed by Nishioka and Wrazidlo, show that the differentiable structures of the spheres and the homology groups of the manifolds (in several classes) are restricted. We see special generic maps are attractive.

Our paper studies the cohomology rings of manifolds admitting such maps. As our new result, we find a new family of manifolds whose cohomology rings are similar and find that the (non-)existence of special generic maps are closely related to the topologies. More explicitly, we have previously found related families and our new manifolds add to these discoveries.



\end{abstract}


\maketitle
\section{Introduction.}
\label{sec:1}
  	According to the so-called Reeb's theorem, Morse functions with exactly two singular points on closed manifolds characterize spheres whose dimensions are not $4$ topologically and the $4$-dimensional unit sphere.
  
  {\it Special generic} maps are, in short, their higher dimensional versions. 
  We first define a {\it special} generic map between smooth manifolds with no boundaries.
  Before that, we introduce fundamental and important terminologies, notions and notation. 
  
  Hereafter, for an integer $k>0$, ${\mathbb{R}}^k$ denotes the $k$-dimensional Euclidean space, which is a smooth manifold canonically and also a Riemannian manifold endowed with the standard Euclidean metric.
  For ${\mathbb{R}}^1$, $\mathbb{R}$ is used in a natural way. This is also a commutative ring and $\mathbb{Z} \subset \mathbb{R}$ denotes the subring of all integers.
  $||x|| \geq 0$ is for the distance between $x$ and the origin $0$ in ${\mathbb{R}}^k$. 
  $S^k\ {\rm (}D^{k+1}{\rm )}:=\{x \in {\mathbb{R}}^{k+1} \mid ||x||=1\ {\rm (}resp.\ ||x|| \leq 1{\rm )}\}$ is the $k$-dimensional unit sphere (resp. ($k+1$)-dimensional unit disk). A canonical projection of the unit sphere $S^k \subset {\mathbb{R}}^{k+1}$ into ${\mathbb{R}}^{k^{\prime}}$ with $k \geq k^{\prime}$ is defined as a map mapping $(x_1,x_2) \in {\mathbb{R}}^{k^{\prime}} \times {\mathbb{R}}^{k+1-k^{\prime}}$ to $x_1 \in {\mathbb{R}}^{k^{\prime}}$. This is also a simplest special generic map. 
  
  Hereafter, $\dim X$ is the dimension of a topological space regarded as a CW complex, which we can define uniquely. A manifold is always regarded as a CW complex. A smooth manifold always has the structure of a so-called {\it PL} manifold in a canonical way and this is a polyhedron.
  For a smooth map $c:X \rightarrow Y$ between smooth manifold, a point $p \in X$ is a {\it singular} point of $c$ if the rank of the differential ${dc}_p$ here is smaller than $\min{\dim X,\dim Y}$. $S(c)$ denotes the set of all singular points of $c$ and we call this the {\it singular set} of $c$.
  
  A {\it diffeomorphism} means a homeomorphism which is a smooth map with no singular points.
  Two smooth manifolds are defined to be {\it diffeomorphic} if a diffeomorphism between the two manifolds exists.  
  A smooth manifold homeomorphic to a sphere is said to be a {\it homotopy sphere}. A homotopy sphere is a {\it standard} (an {\it exotic}) sphere if it is diffeomorphic to a unit sphere (resp. not diffeomorphic to any unit sphere).
  
  \begin{Def}
  	A smooth map $c:X \rightarrow Y$ between manifolds with no boundaries satisfying $\dim X \geq \dim Y$ is said to be {\it special generic} if at each singular point $p$, there exists suitable local coordinates and $f$ is locally
    represented by $(x_1,\cdots,x_{\dim X}) \rightarrow (x_1,\cdots,x_{\dim Y-1},{\Sigma}_{j=1}^{\dim X-\dim Y+1} {x_{\dim Y+j-1}}^2)$.
  \end{Def}

Morse functions in the Reeb's theorem are special generic. As a kind of exercises on smooth manifolds and maps and the theory of Morse functions, we can see that canonical projections of unit spheres are special generic.

  \begin{Prop}
  	The singular set $S(c)$ of the special generic map $c:X \rightarrow Y$ is a {\rm (}$\dim Y-1${\rm )}-dimensional smooth closed submanifold of $X$ and has no boundary. Furtheremore, the restriction $c {\mid}_{S(c)}$ is a smooth immersion,  
  \end{Prop}
     
Other properties are presented in the next section. Studies of Saeki and Sakuma (\cite{saeki1,saeki2,saekisakuma,saekisakuma2,sakuma}) since the 1990s,
  followed by Nishioka (\cite{nishioka}) and Wrazidlo (\cite{wrazidlo,wrazidlo2,wrazidlo3}), show that the differentiable structures of the spheres and the homology groups of the manifolds are restricted in several cases. This explicitly shows that special generic maps are attractive. Some of these results are also presented later.
  
  Our paper studies the cohomology rings of manifolds admitting such maps. We also introduce terminologies and notation on homology groups, cohomology rings, characteristic classes of manifolds such as {\it Stiefel-Whitney classes} and {\it Pontrjagin classes} and other notions from algebraic topology and differential topology in the next section. We also expect that readers have some knowledge on them. \cite{hatcher} and \cite{milnorstasheff} explain about them systematically.
  \begin{MainThm}
  	\label{mthm:1}
  	Let $m \geq 8$ be an arbitrary integer. Let $m^{\prime} \geq 6$ be another integer satisfying the condition $m-m^{\prime} \geq 2$. Assume also the existence of some integer $a$. Furthermore, the following conditions hold.
  	\begin{itemize}
  		\item $a=2$ or $a=m^{\prime}-3$.
  		\item $a+m^{\prime}-1 \leq m$
  		\item $2m^{\prime}-m-a=1$.
  	\end{itemize}
  Let $G$ be an arbitrary finite commutative group which is not the trivial group.
  	Then we have a closed and simply-connected manifold $M$ enjoying the following properties.
  	\begin{enumerate}
  		\item \label{mthm:1.1} There exists an $m^{\prime}$-dimensional closed and simply-connected manifold $M^{\prime}$ and $M$ is diffeomorphic to $M^{\prime} \times S^{m-m^{\prime}}$.
  		  		\item \label{mthm:1.2} The $j$-th homology group of $M^{\prime}$ whose coefficient ring is $\mathbb{Z}$ is isomorphic to $G$ if $j=2,m^{\prime}-3$, $\mathbb{Z}$ if $j=0,m^{\prime}$, and the trivial group, otherwise.
  		\item \label{mthm:1.3} $M$ admits a special generic map into ${\mathbb{R}}^n$ for $m^{\prime}<n\leq m$ whereas it admits no special generic maps into ${\mathbb{R}}^n$ for $1 \leq n \leq 4$ and $n=m^{\prime}$. Furthermore, the special generic maps can be obtained as ones whose restrictions to the singular sets are embeddings and {\rm product organized} ones, defined in Definition \ref{def:2}, later. 
  	  	\end{enumerate}
  \end{MainThm}
  \begin{MainThm}
  	\label{mthm:2}
  	Let $\{G_j\}_{j=1}^5$ be a sequence of free commutative groups such that $G_{j}$ and $G_{6-j}$ are isomorphic for $1 \leq j \leq 2$, that $G_{1}$ is not the trivial group and that the rank of $G_{3}$ is even. Let $G_{{\rm F},1}$ be an arbitrary finite commutative group and $G_{{\rm F},2}$ a finite commutative group which is not the trivial group.
   	Let $\{M_i\}_{i=1}^2$ be a pair of $8$-dimensional closed and simply-connected manifolds enjoying the following properties.
  	\begin{enumerate}
  		\item The $j$-th homology groups of these manifolds whose coefficient rings are $\mathbb{Z}$ are mutually isomorphic to $G_{1} \oplus G_{{\rm F},1}$ for $j=2$, the direct sum of $G_{2} \oplus G_{{\rm F},2}$ for $j=3$, the direct sum of $G_{3} \oplus G_{{\rm F},2}$ for $j=4$, the direct sum of $G_{4} \oplus G_{{\rm F},1}$ for $j=5$, and $G_{5}$ for $j=6$.
  		\item Consider the restrictions to the subgroups generated by all elements whose orders are infinite for these manifolds. They are regarded as subalgebras and isomorphic to the cohomology rings of manifolds represented as connected sums of finitely many manifolds diffeomorphic to the products of two homotopy spheres where the coefficient ring is $\mathbb{Z}$.
  		\item For an arbitrary positive integer $j>0$, the $j$-th Stiefel-Whitney classes and the $j$-th  Pontrjagin classes of these manifolds are the zero elements of the $j$-th cohomology groups whose coefficient rings are $\mathbb{Z}/2\mathbb{Z}$, the ring of order $2$, and the $4j$-th cohomology groups whose coefficient rings are $\mathbb{Z}$, respectively.
  		\item $M_i$ does not admit special generic maps into ${\mathbb{R}}^n$ for $1 \leq n < i+5$ whereas it admits one into ${\mathbb{R}}^n$ for $i+5 \leq n \leq 8$.
  		Furthermore, the special generic maps can be obtained as ones whose restrictions to the singular sets are embeddings and {\rm product organized} ones, defined in Definition \ref{def:2}, later. 
  	\end{enumerate}
  	\end{MainThm}
 Main Theorem \ref{mthm:2} is regarded as a variant of Main Theorem 4 of \cite{kitazawa2}. 
 
 In the next section, we present additional fundamental properties of special generic maps. The third section is devoted to our proof of Main Theorems. We also present existing studies on the the homology groups and the cohomology rings of the manifolds. \\
 \ \\
 {\bf Conflict of Interest.} \\
 The author was a member of the project supported by JSPS KAKENHI Grant Number JP17H06128 "Innovative research of geometric topology and singularities of differentiable mappings"
 (Principal investigator: Osamu Saeki) and 
 is also a member of the project JSPS KAKENHI Grant Number JP22K18267 "Visualizing twists in data through monodromy" (Principal Investigator: Osamu Saeki). The present study is also supported by these projects. \\
 \ \\
 {\bf Data availability.} \\
 Data essentially supporting the present study are all in the present paper.
 
 \section{Fundamental properties and existing studies on special generic maps and the manifolds.}
 
 First, we review elementary algebraic topology and other notions, terminologies and notation we need.

Let $(X,X^{\prime})$ be a pair of topological spaces where these two spaces satisfy the relation $X^{\prime} \subset X$ and may be empty. The ({\it $k$-th}) {\it homology group} of this pair $(X,X^{\prime})$ whose coefficient ring is $A$ is denoted by $H_{k}(X,X^{\prime};A)$. In the case $X^{\prime}$ is empty, we may omit ",$X^{\prime}$" in the notation and we call the homology group (cohomology group) of $(X,X^{\prime})$ the {\it homology group} (resp. {\it cohomology group}) of $X$ in general.

For a topological space $X$ of suitable classes, the {\it $k$-th homotopy group} of $X$ is denoted by ${\pi}_k(X)$.
The class of arcwise-connected cell complex, containing the class of arcwise-connected CW complexes for example, is one of such classes.

Let $(X,X^{\prime})$ and $(Y,Y^{\prime})$ be pairs of topological spaces where these spaces satisfy the relations $X^{\prime} \subset X$ and $Y^{\prime} \subset Y$ and may be empty. Given a continuous map $c:X \rightarrow Y$ satisfying $c(X^{\prime}) \subset Y^{\prime}$, then $c_{\ast}:H_{k}(X,X^{\prime};A) \rightarrow H_{k}(Y,Y^{\prime};A)$, $c^{\ast}:H^{k}(Y,Y^{\prime};A) \rightarrow H^{k}(X,X^{\prime};A)$ and $c_{\ast}:{\pi}_k(X) \rightarrow {\pi}_k(Y)$ denote the canonically induced homomorphisms where in considering homotopy groups the spaces satisfy the suitable conditions for example. 

Let $X$ be a topological space and $A$ a commutative ring. Let
$H^{\ast}(X;A)$ denote the direct sum ${\oplus}_{j=0}^{\infty} H^j(X;A)$ for all integers greater than or equal to $0$. 
For any sequence $\{a_j\}_{j=1}^l \subset H^{\ast}(X;A)$ of elements of length $l>0$, we can define the {\it cup product} ${\cup}_{j=1}^l a_j \in H^{\ast}(X;A)$ for general $l$. $a_1 \cup a_2$ is also used in the case $l=2$. We can have the structure of a graded commutative algebra $H^{\ast}(X;A)$. This is the {\it cohomology ring} of $X$ whose {\it coefficient ring} is $A$.
 
The {\it fundamental class} of a compact, connected and oriented manifold $Y$ is defined as the canonically and uniquely defined element of the ($\dim Y$)-th homology group, which is also a generator of the group $H_{\dim Y}(Y,\partial Y;\mathbb{Z})$, isomorphic to the group $\mathbb{Z}$. 
Let $i_{Y,X}:Y \rightarrow X$ be a smooth immersion satisfying $i_{Y,X}(\partial Y) \subset \partial X$ and $i_{Y,X}({\rm Int}\ Y) \subset {\rm Int}\ X$. In other words, $Y$ is smoothly and {\it properly} immersed or embedded into $X$.

Note that in considering other categories such as the PL category and the topology category, we consider suitable maps playing roles as the smooth immersions play. 
If for an element $h \in H_j(X,\partial X;A)$, the value of the homomorphism ${i_{Y,X}}_{\ast}$ induced by the map $i_{Y,X}:Y \rightarrow X$ at the fundamental class of $Y$ is $h$, then $h$ is {\it represented} by $Y$. 

We also need several fundamental methods and theorems. For example, homology exact sequences for pairs of topological spaces, Mayer-Vietoris sequences, Poincar\'e duality (theorem) for compact and connected (orientable) manifolds, universal coefficient theorem and K\"unneth theorem for the products of topological spaces. 

For such explanations, see \cite{hatcher} again for example.

The {\it diffeomorphism group} of a smooth manifold is the group of all diffeomorphisms from the manifold to itself endowed with the so-called {\it Whitney $C^{\infty}$ topology}.

A {\it smooth} bundle means a bundle whose fiber is a smooth manifold whose structure group is the diffeomorphism group. A {\it linear} bundle means a bundle whose fiber is a Euclidean space, unit sphere, or a unit disk and whose structure group consists of linear transformations. Note that a linear transformation here is defined in a matural and canonical way.

For general theory of bundles, see \cite{steenrod}. \cite{milnorstasheff} is for linear bundles.
\begin{Prop}
	\label{prop:2}
	For a special generic map on an $m$-dimensional closed and connected manifold $M$ into ${\mathbb{R}}^n$, we have the following properties.
	\begin{enumerate}
		\item \label{prop:2.1}
		There exists an $n$-dimensional compact and connected smooth manifold $W_f$ such that a smooth surjection $q_f:M \rightarrow W_f$ and a smooth immersion $\bar{f}:W_f \rightarrow {\mathbb{R}}^n$ enjoying the relation $f=\bar{f} \circ q_f$ exist.
		\item \label{prop:2.2}
		There exists a small collar neighborhood $N(\partial W_f)$ of the boundary $\partial W_f \subset W_f$ and the composition of the restriction of $f$ to the preimage with the canonical projection to $\partial W_f$ gives a linear bundle whose fiber is the {\rm (}$m-n+1${\rm )}-dimensional unit disk $D^{m-n+1}$.
		\item \label{prop:2.3}
		The restriction of $f$ to the preimage of $W_f-{\rm Int}\ N(\partial W_f)$ gives a smooth bundle whose fiber is an {\rm (}$m-n${\rm )}-dimensional standard sphere. It is regarded as a linear bundle in specific cases. For example, in the case $m-n=0,1,2,3$, the bundle is regarded as a linear bundle.
	\end{enumerate}
\end{Prop}

\begin{Prop}
	\label{prop:3}
	In Proposition \ref{prop:2}, we can have an {\rm (}m+1{\rm )}-dimensional compact and connected topological {\rm (}PL{\rm )} manifold $W$ and a continuous {\rm (}resp. piesewise smooth{\rm )} surjection $r:W \rightarrow W_f$ enjoying the following properties where we abuse the notation.
	\begin{enumerate}
		\item $M$ is the boundary of $W$ and the restriction of $r$ to the boundary $M$ is $q_f$.
		\item $r$ gives the structure of a bundle over $W_f$ whose fiber is homeomorphic {\rm (}resp. PL homeomorphic{\rm )} to the {\rm (}$m-n+1${\rm )}-dimensional unit disk $D^{m-n+1}$ and whose structure group consists of piesewise smooth homeomorphisms.
	
		\item $W$ is decomposed into two {\rm (}m+1{\rm )}-dimensional topological {\rm (}resp. PL{\rm )} manifolds $W_1$ and $W_2$.
		\item We can consider the composition of the restriction of $r$ to $W_1$ with a canonical projection to $\partial W_f$. This gives a linear bundle whose fiber is the {\rm (}$m-n+2${\rm )}-dimensional unit disk $D^{m-n+2}$. Furthermore, this bundle is realized as a bundle whose subbundle obtained by considering the boundary of the unit disk and a suitable smoothly embedded copy of the unit disk $D^{m-n+1}$ there as its fiber is the linear bundle of Proposition \ref{prop:2} {\rm (}\ref{prop:2.2}{\rm )}.
		\item The restriction of $r$ to $W_2$ gives a bundle whose fiber is diffeomorphic to the {\rm (}$m-n+1${\rm )}-dimensional unit disk $D^{m-n+1}$ {\rm (}resp. and whose structure group consists of piesewise smooth homeomorphisms{\rm )} . Furthermore, this bundle is realized as a bundle whose subbundle obtained by considering the boundary of the unit disk as its fiber is the smooth bundle of Proposition \ref{prop:2} {\rm (}\ref{prop:2.3}{\rm )}.
		\item In the case $m-n=0,1,2,3$ for example, $W$ can be chosen as a smooth one and we can also do in such a way that $r$ gives a linear bundle over $W_f$.
	\end{enumerate}
\end{Prop}
\begin{Ex}
	\label{ex:1}
	Let $l>0$ be an integer. Let $m \geq n \geq 2$ be integers.
	Other than canonical projections of units spheres, we present simplest special generic maps. We consider a connected sum of $l>0$ manifolds in the family $\{S^{n_j} \times S^{m-n_j}\}_{j=1}^l$ in the smooth category where $1 \leq n_j \leq n-1$ is an integer for $1 \leq j \leq l$. We have a special generic map $f:M \rightarrow {\mathbb{R}}^n$ on the resulting manifold $M$ such that in Proposition \ref{prop:2}, the following properties are enjoyed.
	\begin{enumerate}
		\label{ex:1.1}
		\item $\bar{f}$ is an embedding.
		\item
		\label{ex:1.2}
		 $W_f$ is represented as a boundary connected sum of $l>0$ manifolds each of which is diffeomorphic to each of $\{S^{n_j} \times D^{n-n_j}\}_{j=1}^l$ in the family. The boundary connected sum is considered in the smooth category.
		 \item \label{ex:1.3}
		 The bundles in Proposition \ref{prop:2} (\ref{prop:2.2}) and (\ref{prop:2.3}) are trivial.
	\end{enumerate}  
\end{Ex}
We introduce a result of \cite{saeki1} related to Example \ref{ex:1}.

\begin{Thm}[\cite{saeki1}]
	\label{thm:1}
 An $m$-dimensional closed and connected manifold $M$ admits a special generic map into ${\mathbb{R}}^2$ if and only if either of the following holds.
 \begin{enumerate}
 	\item $M$ is a homotopy sphere which is not an $4$-dimensional exotic sphere.
 	\item A manifold
  represented as a connected sum of smooth manifolds considered in the smooth category where each of the manifolds here is represented as either of the following manifolds. 
\begin{enumerate}
	\item The total space of a smooth bundle over $S^1$ whose fiber is a homotopy sphere for $m \neq 5$.
	\item The total space of a smooth bundle over $S^1$ whose fiber is a standard sphere for $m=5$.
\end{enumerate}
\end{enumerate}
Furthermore, for each manifold here, we can construct a special generic map so that the properties {\rm (}\ref{ex:1.1}{\rm )} and {\rm (}\ref{ex:1.2}{\rm )} in Example \ref{ex:1} are enjoyed with $n_j=1$.
\end{Thm}
\begin{Prop}
	\label{prop:4}
	Let $A$ be a commutative ring.
	In Proposition \ref{prop:2}, the homomorphisms ${q_f}_{\ast}:H_j(M;A) \rightarrow H_j(W_f;A)$, ${q_f}^{\ast}:H^j(W_f;A) \rightarrow H^j(M;A)$ and ${q_f}_{\ast}:{\pi}_j(M) \rightarrow {\pi}_j(W_f)$ are isomorphisms for $0 \leq j \leq m-n$.   
	
\end{Prop}
We can show the following theorem by applying Proposition \ref{prop:4} with Proposition \ref{prop:2}, some fundamental theory on $3$-dimensional compact manifolds and some previously presented facts. For a proof, see the
 original paper \cite{saeki1} for example.
\begin{Thm}[\cite{saeki1}]
	\label{thm:2}
	Let $m>3$ be an arbitrary integer.
	\begin{enumerate}
		\item If an $m$-dimensional closed and simply-connected manifold admits a special generic map into ${\mathbb{R}}^3$, then it is either of the following.
		\begin{enumerate}
		\item A homotopy pshere which is not a $4$-dimensional exotic sphere.
		\item A manifold represented as a connected sum of smooth manifolds considered in the smooth category where each of the manifolds here is represented as the total space of a smooth bundle over $S^2$ whose fiber is a homotopy sphere. Moreover, the homotopy sphere of the fiber is also not a $4$-dimensional exotic sphere.
		\end{enumerate}
	\item 
	Furthermore, for each manifold here admitting a special generic map into ${\mathbb{R}}^3$, we can construct a special generic map so that the properties {\rm (}\ref{ex:1.1}{\rm )} and {\rm (}\ref{ex:1.2}{\rm )} in Example \ref{ex:1} are enjoyed with $n_j=2$.
	\end{enumerate}
\end{Thm}

The following is shown in \cite{nishioka} to show Theorem \ref{thm:3} for $(m,n)=(5,4)$ by applying a classification result of $5$-dimensional closed and simply-connected manifolds of \cite{barden}, Propositon \ref{prop:4} and some additional arguments on homology groups.
\begin{Prop}[\cite{nishioka}]
	\label{prop:5}
	Let $k \geq 3$ be an integer.
For a $k$-dimensional compact and connected manifold $X$ such that $H_1(X;\mathbb{Z})$ is the trivial group, $X$ is orientable and the homology group $H_{j}(X;\mathbb{Z})$ is free for $j=k-2,k-1$. 
	
\end{Prop}

\begin{Thm}[\cite{saeki1} for $(m,n)=(4,3),(5,3),(6,3)$,\cite{nishioka} for $(m,n)=(5,4)$, and \cite{kitazawa5} for $(m,n)=(6,4)$]
\label{thm:3}
	For integers $(m,n)=(4,3),(5,3),(5,4),(6,4)$, An $m$-dimensional closed and simply-connected manifold $M$ admits a special generic map into ${\mathbb{R}}^n$ if and only if $M$ is as follows.
	\begin{enumerate}
	 \item A standard sphere.
	 \item A manifold represented as a connected sum of smooth manifolds considered in the smooth category where each of the manifolds here is as follows.
	 \begin{enumerate}
	 	\item The total space of a linear bundle over $S^2$ whose fiber is the unit sphere $S^{m-2}$.
	 	\item Only in the case $(m,n)=(6,4)$, a manifold diffeomorphic to $S^3 \times S^3$.
	 	\end{enumerate}
 	\end{enumerate}
 	 Furthermore, for each manifold here, we can construct a special generic map so that the properties {\rm (}\ref{ex:1.1}{\rm )} and {\rm (}\ref{ex:1.2}{\rm )} in Example \ref{ex:1} are enjoyed with $n_j=2$ unless $(m,n)=(6,4)$. In the case $(m,n)=(6,4)$, we can do similarly with $n_j=2,3$.
\end{Thm}
\begin{Prop}
	\label{prop:6}
	Let $A$ be a commutative ring.
		In Proposition \ref{prop:2}, the cup product for a finite sequence $\{u_j\}_{j=1}^{l} \subset H^{\ast}(M;A)$ of length $l>0$ consisting of elements whose degrees are at most $m-n$ is the zero element if the sum of all $l$ degrees is greater than or equal to $n$.   
\end{Prop}
We introduce a proof of Propositions \ref{prop:4} and \ref{prop:6} referring to \cite{kitazawa1}.
For terminologies from the PL category and similar categories, see \cite{hudson} for example.
\begin{proof}[A proof of Propositions \ref{prop:4} and \ref{prop:6}]
	This is based on an argument of \cite{saekisuzuoka} for example. We abuse the notation of some Propositions such as Proposition \ref{prop:2} and \ref{prop:3}.
	
	$W$ is an ($m+1$)-dimensional compact and connected (PL) manifold whose boundary is $M$ and collapses to $W_f$, which is an $n$-dimensional compact smooth manifold immersed smoothly into ${\mathbb{R}}^n$. This has the (simple) homotopy type of an ($n-1$)-dimensional polyhedron and more precisely, collapses to an ($n-1$)-dimensional polyhedron. We can see that for $0 \leq j < m+1-(n-1)-1=m-n+1$, the three kinds of homomorphisms in Proposition \ref{prop:4} are isomorphisms. 
	
	In Proposition \ref{prop:6}, the cup product for the finite sequence is regarded as the value of ${q_f}^{\ast}:H^{\ast}(W_f;A) \rightarrow H^{\ast}(M;A)$ at the cup product of a sequence of length $l$ where ${q_f}^{\ast}:H^{\ast}(W_f;A) \rightarrow H^{\ast}(M;A)$ is defined in a canonical way from ${q_f}^{\ast}:H^{i}(W_f;A) \rightarrow H^{i}(M;A)$. This sequence of elements of $H^{\ast}(W_f;A)$ is defined in a canonical and unique way by Proposition \ref{prop:4}. The cup product is the zero element of $H^{\ast}(W_f;A)$. This follows from the fact that $W_f$ has the homotopy type of an ($n-1$)-dimensional polyhedron.
	
	This completes the proof.

\end{proof}
Hereafter, in the present section, we introduce arguments in \cite{kitazawa7} and present some new arguments.

\begin{Prop}[Partially discussed in \cite{kitazawa7}]
	\label{prop:7}
Let $m \geq n \geq 1$ be integers. Let ${\bar{f}}_N:\bar{N} \rightarrow {\mathbb{R}}^n$ be a smooth immersion of an $n$-dimensional compact, connected and orientable manifold $\bar{N}$.
Then there exist a suitable closed and connected smooth manifold $M$ and a special generic map $f:M \rightarrow {\mathbb{R}}^n$ enjoying the following properties
where we abuse the notation in Proposition \ref{prop:2}.
	\begin{enumerate}
	\item $W_f=\bar{N}$ and $\bar{f}={\bar{f}}_N$.
	\item Two bundles in Proposition \ref{prop:2} {\rm (}\ref{prop:2.2}{\rm )} and {\rm (}\ref{prop:2.3}{\rm )} are trivial bundles.
	
	\item Let $f_0:=f$. Then there exist a special generic map $f_j:M \rightarrow {\mathbb{R}}^{n+j}$ for $0 \leq j \leq m-n$ and a smooth immersion $f_j:M \rightarrow {\mathbb{R}}^{n+j}$ for $j \geq m-n$. If $f_N$ is an embedding, then each immersion can be obtained as an embedding. Furthermore, they can be constructed enjoying the relation $f_{j_1}={\pi}_{n+j_2,n+j_1} \circ f_{j_2}$ for any distinct integers $j_1<j_2$ and a canonical projection ${\pi}_{n+j_2,n+j_1};{\mathbb{R}}^{n+j_2} \rightarrow {\mathbb{R}}^{n+j_1}$.
\end{enumerate}
\end{Prop}
\begin{proof}
	Most of our proof is based on a proof in the original paper.
	
	By considering the restriction of the canonical projection of a unit sphere to $\mathbb{R}$ to a suitable hemisphere, we have a Morse function on a copy of the unit disk. This function can be obtained enjoying the following properties.
	\begin{itemize}
		\item It has exactly one singular point and there the function has maximal value (minimal value).
		\item The preimage of the minimal (resp. maximal) value is the boundary of the disk.
		\end{itemize}
	We can prepare the product map of such a function and the identity map on the boundary $\partial \bar{N}$ and we have a surjection onto $N(\partial W_f)$, a suitable collar neighborhood of $\partial \bar{N}$, by restricting the manifold of the target of the product map and composing a suitable embedding. 
	We have a smooth trivial bundle over $\bar{N}-{\rm Int}\ N(\partial W_f)=W_f-{\rm Int}\ N(\partial W_f)$ whose fiber is an ($m-n$)-dimensional standard sphere. By gluing these maps in a suitable way we have a special generic map $f=f_0$ enjoying the first two properties.
	
	\begin{itemize}
		\item The diffeomorphism agreeing with the diffeomorphism used for identification between the connected component in the boundary of $\partial N(\partial W_f)$ and the boundary $\partial W_f -{\rm Int}\ N(\partial W_f)$. 
		\item The identity map on the fiber, which is the unit sphere $S^{m-n}$. Note that the fibers of the both trivial bundles can be regarded as the unit sphere $S^{m-n}$. 
	\end{itemize}

    Let $0 \leq j \leq m-n$.
	We replace the canonical projection of the unit sphere to $\mathbb{R}$ in the former product map by a suitable one onto ${\mathbb{R}}^{1+j}$. For the projection of the latter trivial smooth bundle, we replace the projection by the product map of a canonical projection of the unit sphere to ${\mathbb{R}}^j$ and the identity map on the base space. We glue in a natural and suitable way to have a desired map $f_j$. 
	
	In the case $j>m-n$, for construction of a desired map $f_j$, we consider the canonically defined embeddings of the unit spheres instead. This argument is a new one, presented first in the present paper.
	
	This completes the proof of Proposition \ref{prop:7}, and by the construction, Proposition \ref{prop:8}, which is presented later. 
\end{proof}
\begin{Def}[Partially defined in \cite{kitazawa2}]
	\label{def:2}
	We define the special generic map $f=f_0$ obtained in this way a {\it product-orginized} special generic map.
\end{Def}

\begin{Prop}
	\label{prop:8}
	For $0 \leq j \leq m-n$, $f_j$ in Proposition \ref{prop:7} can be also constructed as a  product-organized special generic map.
\end{Prop}
We can generalize ${\mathbb{R}}^n$ here to a general connected smooth manifold with no boundary. However, we concentrate on the cases of ${\mathbb{R}}^n$ essentially.

For a linear bundle over a base space $X$, we can define the {\it $j$-th Stiefel-Whitney class} as an element of $H^j(X;\mathbb{Z}/2\mathbb{Z})$ and the {\it $j$-th Pontrjagin class} as an element of $H^{4j}(X;\mathbb{Z})$.

The tangent bundles of smooth manifolds are important linear bundles whose fibers are Euclidean spaces. 
 
The {\it $j$-th Stiefel-Whitney class} of a smooth manifold $X$ is defined as that of its tangent bundle and an element of $H^j(X;\mathbb{Z}/2\mathbb{Z})$.
The {\it $j$-th Pontrjagin class} of a smooth manifold $X$ is defined as that of its tangent bundle and an element of $H^{4j}(X;\mathbb{Z})$.

The following follows from elementary arguments related to this.

 \begin{Prop}
 		\label{prop:9}
 		Let $j>0$ be an arbitrary integer.
 	For the manifold $M$ and any positive integer $j$, the $j$-th Stiefel-Whitney class and the $j$-th Pontrjagin class of $M$ are the zero elements of $H^j(M;\mathbb{Z}/2\mathbb{Z})$ and $H^{4j}(M;\mathbb{Z})$, respectively.
 \end{Prop}
 
See \cite{milnorstasheff} for Stiefel-Whitney classes and Pontrjagin classes. 

\section{Main Theorems and related arguments.}
We prove Main Theorems. We need some arguments and results of the author. For example, ones in \cite{kitazawa2,kitazawa3,kitazawa4,kitazawa5,kitazawa6} are important.

\begin{Prop}[\cite{kitazawa2,kitazawa3}]
	\label{prop:10}
	Let $n \geq 5$ be an integer.
	Let $G$ be a finite commutative group. 
We have an $n$-dimensional compact and simply-connected manifold ${\bar{N}}_{G,n}$ enjoying the following properties.
\begin{enumerate}
	\item ${\bar{N}}_{G,n}$ is smoothly embedded in ${\mathbb{R}}^n$.
	\item The boundary $\partial {\bar{N}}_{G,n}$ is connected.
	\item $H_{n-3}({\bar{N}}_{G,n};\mathbb{Z})$ is isomorphic to $G$.
	\item $H_j({\bar{N}}_{G,n};\mathbb{Z})$ is the trivial group for $j \neq 0,n-3$.
\end{enumerate}
\end{Prop}
\begin{proof}
	We prove this referring to \cite{kitazawa3} and in a different way. However, the method is essentially same.
	
	First Let $G$ be a cyclic group.
	We can choose a 3-dimensional closed and connected manifold $_G$ smoothly embedded in ${\mathbb{R}}^n$ such that ${\pi}_1(Y_G)$ and $H_1(Y_G;\mathbb{Z})$ are isomorphic to $G$ and $H_2(Y_G;\mathbb{Z})$ is the trivial group. This is due to \cite{wall} and for example, we can choose a so-called {\it Lens space}.
	
	We remove the interior of a small closed tubular neighborhood $N(Y_G)$ of $Y_G$ and ${S^n}_{N(Y_G)}$ denotes the resulting manifold. $N(Y_G)$ is regarded as the total space of a trivial linear bundle over $Y_G$ whose fiber is the ($n-3$)-dimensional unit disk $D^{n-3}$.
	$\partial {S^{n}}_{N(Y_G)}$ is regarded as the total space of the subbundle of the previous bundle whose fiber is $\partial D^{n-3} \subset D^{n-3}$.
	
	We have a Mayer-Vietoris sequence
	
	$$\rightarrow H_j(\partial {S^n}_{Y_G};\mathbb{Z})=H_j(\partial N(Y_G);\mathbb{Z}) \rightarrow H_j(N(Y_G);\mathbb{Z}) \oplus H_j({S^n}_{N(Y_G)};\mathbb{Z}) \rightarrow H_j(S^n;\mathbb{Z}) \rightarrow$$
	\ \\
	and the homomorphism from the first group to the second group is an isomorphism for $1 \leq j \leq n-1$.
	This is due to the fundamental fact that $H_j(S^n;\mathbb{Z})$ is isomorhic to $\mathbb{Z}$ for $j=0,n$ and the trivial group for $1 \leq j \leq n-1$ and that the homomorphism from $H_n(S^n;\mathbb{Z})$ into $H_{n-1}(\partial {S^n}_{N(Y_G)};\mathbb{Z})$ is an isomorphism between groups isomorphic to $\mathbb{Z}$.
	
	${S^n}_{N(Y_G)} \supset \partial {S^n}_{N(Y_G)}$ are a trivial linear bundle over $Y_G$ whose fiber is the unit disk $D^{n-3}$ and its subbundle obtained by considering the fiber $S^{n-4}=\partial D^{n-3} \subset D^{n-3}$, respectively.
	
	 ${S^n}_{N(Y_G)}$ is a connected manifold.
	 $H_j({S^n}_{N(Y_G)};\mathbb{Z})$ is isomorphic to $H_{n-1}(Y_G;\mathbb{Z})$ for $1 \leq j \leq n-2$ and the trivial group for $j=n-1,n$.
	 
	 ${\pi}_1(\partial {S^n}_{N(Y_G)})$ is isomorphic to the direct sum $\mathbb{Z} \oplus G$ in the case $n=5$ and $G$ in the case $n>5$ and commutative. $S^n$ is simply-connected. For $S^n$ and the same submanifolds $N(Y_G)$ and ${S^n}_{N(Y_G);\mathbb{Z}}$, we can apply Seifert van-Kampen theorem. ${\pi}_1({S^n}_{N(Y_G)})$ is shown to be isomorphic to $\mathbb{Z}$.
	 
	 We define
	 ${\bar{N}}_G$ as a manifold in the following way. First attach a copy of $D^3 \times D^{n-3}$ so that this is the total space of the restriction of the linear bundle $N(Y_G)$ to a copy of the $3$-dimensional unit disk $D^3$ smoothly embedded in the base space $Y_G$. After that eliminate the corner. 
	 
	 We have a Mayer-Vietoris sequence for ${\bar{N}}_G$ and its decomposition into the two manifolds, one of which is ${S^n}_{N(Y_G)}$ and the other of which is the copy of $D^3 \times D^{n-3}$. More precisely, the resulting manifold is decomposed by an ($n-1$)-dimensional smoothly embedded manifold diffeomorphic to $D^3 \times \partial D^{n-3}$ into the two manifold.
	 We can also apply Seifert van-Kampen theorem for these manifolds. 
	 Investigating our Mayer-Vietoris sequence and our argument using Seifert van-Kampen theorem, this completes our proof in the case of a cyclic group $G$.
	 
	 For a general finite commutative group $G$, we have a desired manifold by considering a boundary connected sum of finitely many such manifolds forming the family $\{{\bar{N}}_{G_j}\}_{j=1}^l$ such that each $G_j$ is a cyclic group and that the direct sum ${\oplus}_{j=1}^l G_j$ is isomorphic to $G$. We add that the boundary connected sum is taken in the smooth category. 
	 
	 This completes the proof. 
\end{proof}

We have the following by applying Proposition \ref{prop:10} and other arguments such as`Propositions \ref{prop:7}, \ref{prop:8} and \ref{prop:9}.

\begin{Thm}[\cite{kitazawa5} ($(m,n)=(6,5)$)]
	\label{thm:4}
	Let $m \geq n+1$ be an integer. Let $G$ be a finite commutative group. We have an $m$-dimensional closed and simply-connected manifold $M_{G,n,m-n}$ and a product-organized special generic map $f_{G,n,m-n}:M_{G,n,m-n} \rightarrow {\mathbb{R}}^n$ enjoying the following properties.
	\begin{enumerate}
		\item \label{thm:4.1}
		 For any positive integer $j$, the $j$-th Stiefel-Whitney class and the $j$-th Pontrjagin class of $M_{G,n,m-n}$ are the zero elements of $H^j(M_{G,n,m-n};\mathbb{Z}/2\mathbb{Z})$ and $H^{4j}(M_{G,n,m-n};\mathbb{Z})$, respectively.
		\item \label{thm:4.2}
		 The restriction of the map to the singular set is an embedding. The image is diffeomorphic to ${\bar{N}}_G$ in Proposition \ref{prop:10} and $W_{f_{G,n,m-n}}$ where $W_{f_{G,n,m-n}}$ abuse "$W_f$ in Proposition \ref{prop:2}".
		\item \label{thm:4.3}
		$H_j(M_{G,n,m-n};\mathbb{Z})$ is isomorphic to $G$ for $j=n-3,m-n+2$, $\mathbb{Z}$ for $j=0,m$, and the trivial group otherwise.
	\end{enumerate}
Furthermore, according to \cite{jupp,wall2,zhubr,zhubr2} for example, in a suitable case, the topology and the differentiable structure of the manifold $M_{G,5,1}$ is uniquely defined.
\end{Thm}
\begin{proof}
	
	We have a product-organized special generic map $f_{G,n,m-n}:M_{G,n,m-n} \rightarrow {\mathbb{R}}^n$ on a suitable closed and simply-connected manifold $M_{G,n,m-n}$ enjoying (\ref{thm:4.1}) and (\ref{thm:4.2}) here by Propositions \ref{prop:7}, \ref{prop:8} and \ref{prop:9}
	with Proposition \ref{prop:4}.
	
	It is sufficient to show (\ref{thm:4.3}) here.
	
	$W_{f_{G,n,m-n}}$ is simply-connected and by Proposition \ref{prop:10} collapses to an ($n-2$)-dimensional polyhedron.
	By revising the proof of Propositions \ref{prop:4} and \ref{prop:6}, we can see that ${q_{f_{G,n,m-n}}}_{\ast}:H_j(M_{G,n,m-n};\mathbb{Z}) \rightarrow H_j(W_{f_{G,n,m-n}};\mathbb{Z})$ is an isomorphism for $0 \leq j \leq m-n+1$.
	In the present proof, "Proposition \ref{prop:4}" means this isomorphism. \\
	 \\
	 Hereafter, $W_{G,n,m-n}$ denotes "$W$ in Propositions \ref{prop:2} and \ref{prop:3}" where $M_{G,n,m-n}$ is used instead of "$M$ in these propositions". This is ($m+1$)-dimensional, closed and simply-connected.
	 
		We also have a homology exact sequence
	$$\rightarrow H_{j+1}(W_{G,n,m-n},M_{G,n,m-n};\mathbb{Z}) \rightarrow H_j(M_{G,n,m-n};\mathbb{Z}) \rightarrow H_j(W_{G,n,m-n};\mathbb{Z}) \rightarrow $$
\quad $H_j(W_{G,n,m-n},M_{G,n,m-n};\mathbb{Z}) \rightarrow $ \\
\ \\
	for $(W_{G,n,m-n},M_{G,n,m-n})$.\\
	\ \\
	Case 1 The case $m-n+1 \geq \frac{m}{2}$. \\
	 The resulting manifold $M_{G,n,m-n}$ is a desired manifold
	 by applying Proposition \ref{prop:4} and Poincar\'e duality theorem. \\
	\ \\
	Case 2 The case $m$ is even and $m-n+2=\frac{m}{2}$.\\
	\\
		We also have a homology exact sequence
	$$\rightarrow H_{n-2}(W_{G,n,m-n},M_{G,n,m-n};\mathbb{Z}) \rightarrow H_{n-3}(M_{G,n,m-n};\mathbb{Z}) \rightarrow H_{n-3}(W_{G,n,m-n};\mathbb{Z}) \rightarrow $$
  $H_{n-3}(W_{G,n,m-n},M_{G,n,m-n};\mathbb{Z}) \rightarrow $ \\
  \ \\
	for $(W_{G,n,m-n},M_{G,n,m-n})$. $H_{n-2}(W_{G,n,m-n},M_{G,n,m-n};\mathbb{Z})$ is isomorphic to $H^{m-n+3}(W_{G,n,m-n};\mathbb{Z})$, $H^{m-n+3}(W_{f_{G,n,m-n}};\mathbb{Z})$, and $H_{2n-m-3}(W_{f_{G,n,m-n}},\partial W_{f_{G,n,m-n}};\mathbb{Z})$ by Proposition \ref{prop:3} and by virtue of Poincar\'e duality theorem for $W_{f_{G,n,m-n}}$. We have the relation $2n-m-4=0$ and the previous groups are isomorphic to $H_{1}(W_{f_{G,n,m-n}},\partial W_{f_{G,n,m-n}};\mathbb{Z})$. This is the trivial group since $W_{f_{G,n,m-n}}$ is simply-connected and $\partial W_{f_{G,n,m-n}}$ is connected. This is shown by the homology exact sequence for $(W_{f_{G,n,m-n}},\partial W_{f_{G,n,m-n}})$. 
	Similarly, $H_{n-1}(W_{G,n,m-n},M_{G,n,m-n};\mathbb{Z})$ is isomorphic to $H_{2}(W_{f_{G,n,m-n}},\partial W_{f_{G,n,m-n}};\mathbb{Z})$ and $H^{n-2}(W_{f_{G,n,m-n}};\mathbb{Z})$ by Proposition \ref{prop:3} and by virtue of Poincar\'e duality theorem for $W_{f_{G,n,m-n}}$. This group is shown to be finite by Proposition \ref{prop:10} and universal coefficient theorem. Similarly, $H_{n-2}(W_{G,n,m-n},M_{G,n,m-n};\mathbb{Z})$ is isomorphic to $H_{0}(W_{f_{G,n,m-n}},\partial W_{f_{G,n,m-n}};\mathbb{Z})$ and the trivial group. We have the relation $n-2=\frac{m}{2}$. By Proposition \ref{prop:4}, Proposition \ref{prop:10} and Poincar\'e duality theorem for $M_{G,n,m-n}$, this completes the proof for this case. \\
		\ \\
	Case 3 The case $m$ is odd and $m-n+\frac{3}{2}=\frac{m}{2}$.\\
	\\ We also have a homology exact sequence
	$$\rightarrow H_{n-2}(W_{G,n,m-n},M_{G,n,m-n};\mathbb{Z}) \rightarrow H_{n-3}(M_{G,n,m-n};\mathbb{Z}) \rightarrow H_{n-3}(W_{G,n,m-n};\mathbb{Z}) \rightarrow $$
$H_{n-3}(W_{G,n,m-n},M_{G,n,m-n};\mathbb{Z}) \rightarrow $ \\
\ \\
	for $(W_{G,n,m-n},M_{G,n,m-n})$. $H_{n-2}(W_{G,n,m-n},M_{G,n,m-n};\mathbb{Z})$ is isomorphic to $H^{m-n+3}(W_{G,n,m-n};\mathbb{Z})$, $H^{m-n+3}(W_{f_{G,n,m-n}};\mathbb{Z})$, and $H_{2n-m-3}(W_{f_{G,n,m-n}},\partial W_{f_{G,n,m-n}};\mathbb{Z})$ by Proposition \ref{prop:3} and by virtue of Poincar\'e duality theorem for $W_{f_{G,n,m-n}}$. We have the relation $2n-m-3=0$ and the previous groups are isomorphic to $H_{0}(W_{f_{G,n,m-n}},\partial W_{f_{G,n,m-n}};\mathbb{Z})$ and the trivial group. We have the relation $n-2=\frac{m-1}{2}$. By Proposition \ref{prop:4}, Proposition \ref{prop:10}  and Poincar\'e duality theorem for $M_{G,n,m-n}$, this completes the proof for this case. \\
	\ \\	
	This completes the proof of (\ref{thm:4.3}). \\
	This completes the proof.
	\end{proof}
\begin{Rem}
	\cite{nishioka} has studied Case 3 of the proof of Theorem \ref{thm:4}. More precisely, for a positive integer $k>0$ and  a ($2k+1$)-dimensional closed and connected manifold $M$ such that the group $H_1(M;\mathbb{Z})$ is the trivial group admitting a special generic map into ${\mathbb{R}}^{k+2}$, the group $H_k(M;\mathbb{Z})$ has been shown to be free. The case $(m,n)=(5,4)$ of Theorem \ref{thm:3} concerns the case $k=2$ here.
\end{Rem}

\begin{proof}
[A proof of Main Theorem \ref{mthm:1}]
First consider an $m^{\prime}$-dimensional closed and simply-connected manifold $M_{G,5,m^{\prime}-5}$ in Theorem \ref{thm:4} where $m^{\prime}$ denotes "$m$ of Theorem \ref{thm:4}". We can embed this smoothly into ${\mathbb{R}}^{m^{\prime}+1}$ by Proposition \ref{prop:7} for example.

We define $M$ as $M:=M_{G,5,m^{\prime}-5} \times S^{m-m^{\prime}}$. 
This completes our exposition on the properties (\ref{mthm:1.1}) and (\ref{mthm:1.2}).

By Propositions \ref{prop:2}, \ref{prop:4} and \ref{prop:5} for example, this does not admit special generic maps into ${\mathbb{R}}^n$ for $n=1,2,3,4$. By the previous argument, we can consider the product map of a Morse function with exactly two singular points on $S^{m-m^{\prime}}$ and the identity map on $M_{G,5,m^{\prime}-5}$ and embed the image smoothly in a suitable way into ${\mathbb{R}}^{m^{\prime}+1}$. This can be constructed as product-organized one. This implies the existence of special generic maps into ${\mathbb{R}}^n$ for $m^{\prime}<n \leq m$.

We prove the non-existence into  ${\mathbb{R}}^{m^{\prime}}$. Suppose that such a special generic map $f$ on $M$ into ${\mathbb{R}}^{m^{\prime}}$ exists. From Proposition \ref{prop:4}, the homomorphism ${q_f}_{\ast}$ maps the homology group $H_{m-m^{\prime}}(M;\mathbb{Z})$ onto $H_{m-m^{\prime}}(W_f;\mathbb{Z})$ as an isomorphism. 
$m^{\prime}-1$ is the dimension of the boundary $\partial W_f \subset W_f$ and the relation $a+(m^{\prime}-1) \leq m$ is assumed as a condition. We can apply Poincar\'e duality or intersection theory for $M$. From this,  
the torsion subgroup of $H_{a}(M;\mathbb{Z})$, which is not the trivial group from the assumption and K\"unneth theorem for the product $M_{G,5,m^{\prime}-5} \times S^{m-m^{\prime}}$, is mapped onto $H_{a}(W_f;\mathbb{Z})$ by the homomorphism ${q_f}_{\ast}$ and this is an isomorphism.

Here, we apply fundamental methods used first in \cite{kitazawa4}, generalizing some methods used in \cite{saeki1}, and used also in \cite{kitazawa5,kitazawa6} for example.

We choose an element of $h_1 \in H_{m-m^{\prime}}(M;\mathbb{Z})$ and an element of $h_2 \in H_{a}(M;\mathbb{Z})$ which are not the zero elements. Furthermore, we choose $h_2$ as an element of a summand of a direct sum decomposition of the torsion subgroup of $H_{a}(M;\mathbb{Z})$ into cyclic groups. This decomposition is due to fundamental theorem on the structures of finitely generated commutative groups. Note that $a=2,m^{\prime}-3$ and that $M$ is the product of a manifold $M_{G,5,m^{\prime}-5}$ in Theorem \ref{thm:4} and $S^{m-m^{\prime}}$.
By the isomorphisms before and Poincar\'e duality theorem for $W_f$, we have an element of $H_{m^{\prime}-(m-m^{\prime})}(W_f,\partial W_f;\mathbb{Z})$ which is not the zero element corresponding to the element ${q_f}_{\ast}(h_1)$, the value of the isomorphism presented before at $h_1$, in a canonical way uniquely. In a similar reason, there exists a suitable commutative group $G_{h_2}$ which is finite and cyclic and which is not the trivial group and we have an element of $H_{m^{\prime}-a}(W_f,\partial W_f;G_{h_2})$ which is not the zero element corresponding to the element of $H_{a}(W_f;G_{h_2})$, obtained by mapping $h_2$ by the isomorphism defined canonically by ${q_f}_{\ast}$ before and mapping by the canonically defined homomorphism defined naturally from the canonical quotient map from $\mathbb{Z}$ onto $G_{h_2}$, in a canonical way uniquely. For $h_1$, we consider the value of the canonically defined homomorphism defined naturally from the canonical quotient map from $\mathbb{Z}$ onto $G_{h_2}$ at the obtained element of $H_{m^{\prime}-(m-m^{\prime})}(W_f,\partial W_f;\mathbb{Z})$.
We put these two elements by ${h_{1,G_{h_2}}}^{\prime} \in H_{m^{\prime}-(m-m^{\prime})}(W_f,\partial W_f;G_{h_2})$ and ${h_{2,G_{h_2}}}^{\prime} \in H_{m^{\prime}-a}(W_f,\partial W_f;G_{h_2})$. By considering a (so-called {\it generic}) {\it intersection}, we have an element of $H_1(W_f;\partial W_f;G_{h_2})$. The degree $1$ is due to the relation $(m^{\prime}-(m-m^{\prime}))+(m^{\prime}-a)-m^{\prime}=2m^{\prime}-m-a=1$ with the assumption on the integers. 
This is the sum of elements represented by closed intervals embedded smoothly and properly in $W_f$. 
Furthermore, two boundary points of the each interval is mapped into distinct connected components of $\partial W_f$ and the interior is embedded into the interior of $W_f$.
Related to this, circles in (the interior of) $W_f$ are null-homotopic since $W_f$ is simply-connected.
By respecting the bundle whose projection is $r:W \rightarrow M$ in Proposition \ref{prop:3} where we abuse the notation and considering a variant of so-called {\it Thom isomorphisms} or {\it prism-operators}, we have an element of $H_{m-m^{\prime}+2}(W,M;G_{h_2})$. By considering the boundary, we have an element of $H_{m-m^{\prime}+1}(M;G_{h_2})$.
For this see \cite{saeki1} and for algebraic topological notions see \cite{hatcher,milnorstasheff} for example.
 
We go back to our arguments for the proof. If this element is not the zero element, then by fundamental arguments on Poincar\'e duality and intersection theory, this element must be obtained as an element which is the value of the canonically defined homomorphism induced naturally from the canonical quotient map from $\mathbb{Z}$ to $G_{h_2}$. By the structure of $M=M_{G,5,m^{\prime}-5} \times S^{m-m^{\prime}}$ and conditions on homology groups together with Poincar\'e duality and K\"unneth theorem for example, this is not the zero element and this is not an element which is the value of the canonically defined homomorphism induced naturally from the canonical quotient map from $\mathbb{Z}$ to $G_{h_2}$. 
We add that this homomorohism is one from $H_{m-m^{\prime}+1}(M;\mathbb{Z})$ into $H_{m-m^{\prime}+1}(M;G_{h_2})$ as a precise exposition.

We find a contradiction. We have (\ref{mthm:1.3}). This completes the proof.

\end{proof}
\begin{proof}[A proof of Main Theorem \ref{mthm:2}]

	$M_1$ is obtained as a connected sum of the following three manifolds taken in the smooth category. 
	\begin{itemize}
		\item A $8$-dimensional manifold 
	 admitting a product-organized special generic map into ${\mathbb{R}}^5$ in Theorem \ref{thm:4} by considering "$(m,n,G)=(8,5,G_{{\rm F},1})$ in the theorem". 
	 \item A $8$-dimensional manifold 
	 admitting a product-organized special generic map into ${\mathbb{R}}^6$ in Theorem \ref{thm:4} by considering "$(m,n,G)=(8,6,G_{{\rm F},2})$ in the theorem".
	 \item A manifold represented as a suitably chosen connected sum of manifolds diffeomorphic to $S^2 \times S^6$, $S^3 \times S^5$ or $S^4 \times S^4$, which admit special generic maps into ${\mathbb{R}}^6$ in Example \ref{ex:1}. The special generic maps discussed in Example \ref{ex:1} are also constructed as product-organized maps by fundamental arguments on special generic maps.
	 \end{itemize}
	We can see the non-existence of special generic map on $M_1$ into ${\mathbb{R}}^n$ for $n=1,2,3,4,5$ from Propositions \ref{prop:2}, \ref{prop:4}, \ref{prop:5} and Theorem \ref{thm:1} for example. 
	By a fundamental argument for construction of special generic maps in \cite{saeki1}, we can construct a product-organized special generic map on $M_1$ into ${\mathbb{R}}^6$.
	
	We give $M_2$. This is obtained as a connected sum of the following three manifolds taken in the smooth category.
	\begin{itemize}
		\item A $8$-dimensional manifold 
	admitting a product-organized special generic map into ${\mathbb{R}}^5$ in Theorem \ref{thm:4} by considering "$(m,n,G)=(8,5,G_{{\rm F},1})$ in the theorem". 
		\item 
		 A $8$-dimensional closed and simply-connected manifold "$M^{\prime} \times S^2$ in Main Theorem \ref{mthm:1}", having the 3rd homology group $H_3(M^{\prime} \times S^3;\mathbb{Z})$ isomorphic to $G_{{\rm F},2}$. More precisely, we consider the case "$(m,m^{\prime},a,G)=(8,6,3,G_{{\rm F},2})$ in Main Theorem \ref{mthm:1}".
		 \item A manifold represented as a suitably chosen connected sum of manifolds diffeomorphic to $S^2 \times S^6$, $S^3 \times S^5$ or $S^4 \times S^4$, which admit special generic maps into ${\mathbb{R}}^7$ in Example \ref{ex:1}. These maps can be constructed as product-organized maps as before.
		\end{itemize} 
		
We can see the non-existence of special generic maps on $M_2$ into ${\mathbb{R}}^n$ for $n=1,2,3,4,5,6$. For $n=1,2,3,4,5$, for example, Propositions \ref{prop:2}, \ref{prop:4}, \ref{prop:5}, \ref{prop:6} and Theorem \ref{thm:1} complete the proof. For $n=6$ here, we apply Main Theorem \ref{mthm:1} directly or revise arguments in the proof of Main Theorem \ref{mthm:1} suitably if we need. In the case ($m=8$ and) $n=6$, $m-n=8-6=2$ and we do not have any sequence of elements of a finite length of $H^j(M;A)$ with $0 \leq j \leq m-n=2$ the cup product for which is not the zero element whose degree is greater than or equal to $6$. This means that we cannot apply Proposition \ref{prop:6} to give a proof for the case $n=6$. We can show the existence of special generic maps into ${\mathbb{R}}^n$ for $n=7,8$ similarly.

This completes the proof.
\end{proof}

This adds to Main Theorem 4 of \cite{kitazawa2} where situations are a bit different.
Last, for example, our new examples are discovered especially motivated by \cite{kitazawa6}. Compare these results to our new results.


\end{document}